\documentclass{amsart}
\usepackage{amsmath,amsthm,amssymb}

\usepackage{mathrsfs}
\usepackage{bbm}
\usepackage{ifthen}
\numberwithin{equation}{section}

\newtheorem{thm}{Theorem}[section]

\newtheorem{prop}[thm]{Proposition}
\newtheorem{lem}[thm]{Lemma}
\newtheorem{example}[thm]{Example}
\theoremstyle{remark}
\newtheorem{rem}[thm]{Remark}
\theoremstyle{definition}
\newtheorem{defn}[thm]{Definition}

\renewcommand{\Re}{\mathrm{Re}\,}
\newcommand{\eps}{\varepsilon}
\newcommand{\sgn}{\mathrm{sgn}}
\newcommand{\one}{\mathbbm{1}}
\newcommand{\weak}{\rightharpoonup}
\newcommand{\norm}[1]{\ifthenelse{\equal{#1}{}}{\mbox{$\|\cdot\|$}}{\mbox{$\| #1 \|$}}}
\newcommand{\dual}[2]{\ifthenelse{\equal{#1}{}}{\mbox{$ \left\langle \,\cdot\, , \, \cdot \, \right\rangle $}}{
\mbox{$ \left\langle #1 \,  , \, #2 \right\rangle$}}}
\newcommand{\rg}{\mathrm{rg}}
\newcommand{\CR}{\mathbb{R}}
\newcommand{\CC}{\mathbb{C}}
\newcommand{\CN}{\mathbb{N}}

\newcommand{\cF}{\mathscr{F}}
\newcommand{\minus}{\,\mbox{-}\,}

\begin{document}

\title{A Pettis-Type Integral and Applications to Transition Semigroups}

\author{Markus Kunze}
\address{Delft Institute of Applied Mathematics, Delft University of Technology, 
P.O. Box 5031, 2600 GA Delft, The Netherlands}
\email{M.C.Kunze@tudelft}
\thanks{The author was supported partially by the Deutsche Forschungsgesellschaft 
in the framework of the DFG research training group 1100 at the University of Ulm
and partially by VICI subsidy 639.033.604 in the `Vernieuwingsimpuls' programme of the Netherlands Organization
for Scientific Research (NWO).}

\abstract
Motivated by applications to transition semigroups, we introduce the notion of a norming dual pair
and study a Pettis-type integral on such pairs. In particular, we establish a sufficient condition
for integrability. We also introduce and study a class of semigroups on such dual pairs which are an abstract 
version of transition semigroups. Using our results, we give conditions ensuring that a semigroup
consisting of kernel operators has a Laplace transform which also consists of kernel operators.
We also provide conditions under which a semigroup is uniquely determined by its Laplace transform.
\endabstract



\maketitle

\section{Introduction}
In a certain way, the Bochner integral is the appropriate generalization of the Lebesgue integral to
the Banach space setting. The criterion for Bochner integrability is fairly easy: a strongly measurable
function is Bochner integrable if and only if its norm is integrable.

However, not for all applications this notion of integrability is suitable. In this case, one can 
sometimes resort to the Pettis integral, see \cite{pettis38} or Section II.3 of \cite{diesteluhl}, which still 
yields a rich theory. 
But even the notion of weak measurability, which is a prerequisite for Pettis integrability, is often 
too strong. Indeed, already in the simple example of the shift semigroup on the space of bounded
Borel measures on the real line, the orbits of the semigroup are not weakly measurable, cf.\ \cite{feller}.

It is thus natural to replace in the definition of the Pettis integral
the dual $X^*$ of a Banach space $X$ with some subset $Y$ of $X^*$. This leads to the notion of
$Y$-integrability, see \cite{musial}. However, except for the special cases of the Pettis integral and
the $\mathrm{weak}^*$-integral, this notion of integrability was not the subject of broad investigation.\medskip

In Section \ref{sect3}, we will study $Y$-integrability in the case where $Y$ is a norm-closed subspace
of $X^*$ which is norming for $X$. In particular, we prove a sufficient condition for $Y$-integrability
(Theorem \ref{t.criteria}). Our main assumption in that theorem is the existence of a quasi-complete, consistent
topology $\tau$ on $X$.

It should be noted that the $Y$-integral actually coincides with the Pettis integral 
on the locally convex space $(X,\tau)$, for any locally convex topology $\tau$ on $X$ such that
$(X,\tau )'=Y$. However, in contrast to conditions for Pettis integrability on locally convex
spaces, see \cite{jach}, our condition is more in the spirit of the characterization of Bochner integrability
and the proof makes extensive use of the norm topologies on $X$ and $Y$.\medskip

At first sight, the notion of $Y$-integrability seems quite technical and arbitrary since there is no canonical
choice for the space $Y$. This might be the reason why this notion of integrability was not studied
in more detail so far. However, in applications to transition semigroups, it is quite clear which
space $Y$ should be chosen. This example serves both as a motivation and as an application of our theory.\medskip

We recall that associated to a Markov process $(X_t)_{t\geq 0}$, 
taking values in a measurable space $(E,\Sigma)$,  there are, in fact, {\it two} 
semigroups which both have a stochastic interpretation 
and which are connected to each other by duality. Namely, there is a semigroup $\mathbf{T}$ 
on the space $B_b(E)$ of all bounded measurable functions on $E$, used to compute conditional expectations, 
and a second semigroup $\mathbf{T}'$ acting on $\mathcal{M}(E)$, the space of all bounded measures on 
$E$, which gives the distributions of the random variables $X_t$. This duality
relation actually characterizes the operators in $\mathbf{T}$ as kernel operators, cf.\
\cite{lanthi} and Section \ref{sect2}. Furthermore, this duality relation suggests to replace the space $B_b(E)^*$, 
which is usually quite large, with $\mathcal{M}(E)$ for the purpose of integrating the orbits of $\mathbf{T}$ and, similarly, 
to replace $\mathcal{M}(E)^*$ with $B_b(E)$ to integrate the orbits of $\mathbf{T}'$.\medskip

In applications, it is also important to replace $B_b(E)$ with some closed subspace $X$ of $B_b(E)$ which
is invariant under the semigroup. For example, if $E$ is a topological
space, one wants to replace $B_b(E)$ with $X=C_b(E)$. If $E$ is additionally locally compact, one wants to
work on the space $X= C_0(E)$. This is the classical example of a \emph{Feller semigroup}.
In order to treat also these situations, we shall work on a general norming dual pair $(X,Y)$, 
see Definition \ref{d.ndp}, and introduce in Section \ref{sect4} the abstract notion of  
`semigroups on norming dual pairs'. 

We call such a semigroup `integrable' if it is possible to compute the Laplace transform
in an appropriate way and obtain again operators which respect the duality. {\sc Jefferies} \cite{jefferies1, jefferies2}
studied weakly integrable semigroups on locally convex spaces and made similar assumptions on the semigroup. However,
he does not assume that the Laplace transform respects the duality.
In Theorem \ref{t.measurabledescription}, 
we will show that this assumption -- actually it suffices to consider the Laplace integral in a single point --
is equivalent to the requirement that all orbits of $\mathbf{T}$ are locally $Y$-integrable and all 
orbits of $\mathbf{T}'$ are locally
$X$-integrable. Using the results of Section \ref{sect3}, we study integrable transition semigroups
on the space $C_b(E)$ in Section \ref{sect5}.\medskip

In order to treat transition semigroups on $C_b(E)$, several approaches have been proposed 
in the literature. We mention
the theory of weakly continuous semigroups of {\sc Cerrai} \cite{cerrai}, the theory of bi-continuous semigroups
of {\sc K\"uhnemund} \cite{kuehne}, see also \cite{farkas, kvn} for applications in the context of transition semigroups, and the theory of $\pi$-semigroups by {\sc Priola} \cite{pri}. It should be noted that in these approaches
additional assumptions, in particular continuity and equicontinuity assumptions, are made which ensure that a 
Riemann integral can be used to compute the Laplace transform. Even though integrability is not an issue
in these approaches, the question remains in which sense the Laplace transform determines the semigroup. 
More precisely: is it possible that there exists a second semigroup (not necessarily satisfying the continuity and
equicontinuity assumptions) which yields the same Laplace transform? 
This is not the case as our uniqueness theorem (Theorem \ref{t.laplaceunique}) shows.

It is also possible to interpret the continuity and equicontinuity assumptions in the articles mentioned above
from our `dual point of view'. This yields interesting new results for such semigroups which are
presented elsewhere \cite{k09b}.

\subsection*{Notation}
If $(E,\Sigma )$ is a measurable space, then $B_b(E)$ denotes the space of all bounded,
measurable functions $f: E \to \CC$, endowed with the supremum norm. 
By $\mathcal{M}(E)$ we denote the space of all complex measures on $(E,\Sigma )$. 
The {\it total variation} of a measure $\mu$ is defined by
\[ |\mu |(A) = \sup_{\mathcal{Z}}\sum_{B\in\mathcal{Z}}|\mu (B)| \,\, ,\]
where the supremum is taken over all partitions $\mathcal{Z}$ of $A$ into finitely many,
disjoint, measurable sets. Endowed with the total variation norm $\norm{\mu} :=
|\mu |(E)$, the space $\mathcal{M}(E)$ is a Banach space. 

Now suppose that $E$ is a topological space. 
Then $C_b(E)$ denotes the Banach space of all bounded, continuous functions 
$f : E \to \CC$. The Borel 
$\sigma$-algebra of $E$ is denoted by $\mathcal{B}(E)$. If we speak about measures or measurable functions on a topological space, this is always to be understood with respect to the Borel $\sigma$-algebra. 
A positive measure $\mu \in \mathcal{M}(E)$ is a {\it Radon measure}
if $\mu (A) = \sup\{\mu (K) \, : \, K \subset A\,  , \,K \,\, \mbox{compact}\, \}$ for
all $A \in \mathcal{B}(E)$. An arbitrary $\mu \in \mathcal{M}(E)$ is called a Radon measure
if $|\mu |$ is a Radon measure. We denote the space of all Radon measures on $E$
by $\mathcal{M}_0(E)$. This is a closed subspace of $\mathcal{M}(E)$.

By $\one_A$ we denote the characteristic function of a set $A$. For a complex number $z$,
$\sgn z$ denotes the signum of $z$, i.e.\ $\sgn z := |z|^{-1}\bar{z}$ if $z \neq 0$ and $\sgn 0 := 0$. 
The Dirac measure in a point $x$ is denoted by $\delta_x$. If $E$ is a metric space, then $B(x,r)$ denotes
the open ball of radius $r$ centered at $x$ and $\overline{B}(x,r )$ denotes the closure of
that ball. If $X$ is a Banach space, then $X^*$ denotes the norm dual of $X$ and $\dual{}{}_*$
denotes the canonical duality between $X$ and $X^*$.

\section{Norming Dual Pairs}\label{sect1}

\begin{defn}\label{d.ndp}
Let $X$ and $Y$ be nontrivial Banach spaces and $\dual{}{}$ be a 
duality pairing between $X$ and $Y$. 
Then $(X,Y,\dual{}{})$ is called a {\it norming dual pair}, if
\[ \norm{x}_X = \sup\{ |\dual{x}{y}|\, : \, y \in Y\, , \, \norm{y}_Y \leq 1 \}\,\,\phantom{.} \]
and
\[ \norm{y}_Y = \sup\{ |\dual{x}{y}|\, : \, x \in X\, , \, \norm{x}_X \leq 1 \} \,\, .\]
\end{defn}

We will write $(X,Y)$ for a norming dual pair if 
the duality pairing is understood. Note that if $(X,Y)$ is a norming dual
pair, then so is $(Y,X)$.

As we have done already in the introduction, we will frequently consider $Y$ as a closed subspace
of $X^*$ via $\dual{x}{y}_*= \dual{x}{y}$. With this interpretation, 
$(X,Y)$ is a norming dual pair if and only if $Y$ is a 
closed subspace of $X^*$ which is norming for $X$ in the sense of \cite{arni}. 
For $Y\subset X^*$ to be norming for $X$ it is necessary that 
$Y$ is $\mathrm{weak}^*$-dense in $X^*$. However, 
not every $\mathrm{weak}^*$-dense, closed subspace of $X^*$ is norming, see \cite{dl}.

\begin{example}
If $X$ is a Banach space, then $(X,X^*)$, and thus by symmetry also
$(X^*,X)$, is a norming dual pair with the canonical duality $\dual{}{}_*$. 
If $X$ is reflexive, then $Y = X^*$ is the only closed subspace of $X^*$
such that $(X,Y)$ is a norming dual pair. Indeed, if $Y \subset X^*$ is norm-closed, 
it follows from the Hahn-Banach theorem that $Y$ is weakly closed and hence, 
by reflexivity, weak${}^*$-closed. Since $Y$ is weak${}^*$-dense, it follows that $Y = X^*$.
\end{example}

\begin{example}
Let $(E, \Sigma )$ be a measurable space. Then $(B_b(E), \mathcal{M}(E))$
is a norming dual pair with respect to the duality $\dual{}{}$, given by
\begin{equation}\label{eq.duality}
\dual{f}{\mu} := \int_E f \, d\mu \,\, . 
\end{equation}
\end{example}

\begin{proof}
We clearly have $|\int f \, d\mu | \leq \norm{f}_{\infty}\cdot \norm{\mu}$.
Considering Dirac measures, we obtain $\norm{f}_{\infty} = \sup\{
|\dual{f}{\mu}| \, : \, \mu \in \mathcal{M}(E)\, , \, \norm{\mu} \leq 1\}$. Now let
$\mu \in \mathcal{M}(E)$. If $\mathcal{Z}$ is a partition of $E$ into finitely many, pairwise
disjoint, measurable sets, then $f_{\mathcal{Z}} := \sum_{A\in \mathcal{Z}} \sgn \mu (A) \cdot
\one_A$ is a measurable function of norm at most 1. Furthermore, $|\dual{f_{\mathcal{Z}}}{\mu}|
= \sum_{A\in \mathcal{Z}}|\mu (A)|$. Taking the supremum over all such partitions $\mathcal{Z}$, it
follows that $(B_b(E),\mathcal{M}(E))$ is a norming dual pair.
\end{proof}

\begin{example}
Let $E$ be a completely regular Hausdorff space. Then, endowed with the duality \eqref{eq.duality},
$(C_b(E), \mathcal{M}_0(E))$ is a norming dual pair. 
\end{example}

For a complete, separable metric space $E$, the proof of this statement is implicitly contained
in the proof of Theorem 2.2 of \cite{pri}. We give a proof in the general case.

\begin{proof}
It suffices to show that $\norm{\mu} \geq \sup\{|\dual{f}{\mu}|\, : \, f \in C_b(E)\, , \,
\norm{f}_{\infty} \leq 1 \}$. Let $\mu \in \mathcal{M}_0(E)$ be fixed and  $\mathcal{Z} = 
\{A_1, \ldots , A_n\}$ be a finite partition of $E$ into measurable sets. Since $\mu$ is a Radon
measure, given $\eps > 0$, we find compact sets $C_k \subset A_k$ for $k = 1, \ldots , n$ such that
$|\mu (A_k) - \mu (C_k)| \leq |\mu |(A_k\setminus C_k) \leq \frac{\eps}{n}$. As $E$ is completely
regular, there exists a continuous function $f: E \to \CC$ such that $\norm{f}_\infty\leq 1$ and $f|_{C_k}
\equiv \sgn \, \mu (C_k)$. Now
\[
\left| \int f \, d\mu \right| 
 \geq  \sum_{k=1}^n|\mu (C_k)| - \sum_{k=1}^n|\mu| (A_k\setminus C_k)
\geq  \sum_{k=1}^n |\mu (A_k)| - 2\eps 
\]
follows from the reverse triangle inequality.
As $\eps$ was arbitrary, $\sum_{k=1}^n|\mu (A_k)| \leq \sup\{|\dual{f}{\mu}|\, : \, f \in C_b(E)\, , \,
\norm{f}_{\infty} \leq 1\, \}$. Taking the supremum over all such partitions $\mathcal{Z}$ of $E$,
the claim follows.
\end{proof}

In what follows, we will be interested in locally convex topologies $\tau$ on $X$
which are {\it consistent (with the duality)}. By this we mean that $(X,\tau )'
= Y$, i.e. every $\tau$-continuous linear functional $\varphi$ on $X$ is of the
form $\varphi (x) = \dual{x}{y}$ for some $y \in Y$. By the Mackey-Arens theorem, see
\cite[21.4 (2)]{koethe}, a consistent topology is finer than the {\it weak topology} $\sigma (X,Y)$
and coarser than the {\it Mackey topology} $\mu (X,Y)$. To simplify notation, we will write
$\sigma$ for $\sigma (X,Y)$ and $\sigma'$ for the $\sigma (Y,X)$ topology on
$Y$. We will write $\weak$ (resp. $\weak'$) to indicate convergence with respect to
$\sigma$ (resp. $\sigma'$). We will use the name of a topology as a label or prefix to topological 
notions to indicate
that it is to be understood with respect to that topology. Without label or prefix, such notions
are always understood with respect to the norm topology.
 
We now characterize bounded subsets in a norming dual pair.

\begin{prop}\label{p.bounded}
Let $(X,Y)$ be a norming dual pair and $\tau$ be a consistent topology on $X$. 
For a subset $M\subset X$, the following are equivalent.
\begin{enumerate}
\item $M$ is norm-bounded;
\item $M$ is $\sigma$-bounded;
\item $M$ is $\tau$-bounded.
\end{enumerate}
\end{prop}

\begin{proof}
(i) $\Rightarrow$ (ii). As $M$ is $\sigma$-bounded iff $\sup_{x\in M} |\dual{x}{y}| < \infty$
for all $y \in Y$, this implication is trivial. $\quad$ (ii) $\Rightarrow$ (i). If $M$ is $\sigma$-bounded,
the uniform boundedness principle in $Y^*$ implies that $\sup_{x\in M}\norm{x} =\sup_{x\in M} \norm{x}_{Y^*}$ is bounded.
$\quad$ (ii) $\Leftrightarrow$ (iii). See \S 20.11 (7) in \cite{koethe}.
\end{proof}

\section{Operators on Norming Dual Pairs}\label{sect2}

If $\tau$ is a locally
convex topology on $X$, we denote the  algebra of $\tau$-continuous linear operators
on $X$ by $L(X,\tau )$. For $\tau =\tau_{\|\cdot\|}$, where $\tau_{\|\cdot\|}$ is the norm topology,
 we merely write $L(X)$ instead of $L(X,\tau_{\|\cdot\|})$.
 For $T \in L(X)$, we denote its norm-adjoint by $T^*$.
If $T \in L(X,\sigma )$, then we denote its $\sigma$-adjoint by $T'$.

\begin{prop}\label{p.operator}
Let $(X,Y)$ be a norming dual pair. 
\begin{enumerate}
 \item $T\in L(X,\sigma )$ if and only if
$T\in L(X)$ and $T^*Y \subset Y$. In this case, $T' = T^*|_Y$.
Furthermore, $\norm{T}_{L(X)} = \norm{T'}_{L(Y)}$.
 \item $L(X,\sigma )$ is closed in $L(X)$ with respect to the 
operator norm. 
\end{enumerate}
\end{prop}

\begin{proof}
(i) If $T$ is $\sigma$-continuous, then $T$ maps $\sigma$-bounded sets into $\sigma$-bounded sets.
By Proposition \ref{p.bounded}, $T$ is a bounded operator on $X$, hence $T\in L(X)$. Furthermore,
as $T$ is $\sigma$-continuous, it has a $\sigma$-adjoint $S$. But for $x \in X$ and $y\in Y$
we have $\dual{Tx}{y} = \dual{x}{Sy} = \dual{x}{T^*y}_*$. It follows that $T^*y = Sy \in Y$, 
and thus $T^*$ leaves $Y$ invariant.
$\quad$
Conversely assume that $T\in L(X)$ and $T^*Y\subset Y$. Then we have $\dual{Tx}{y} = \dual{x}{T^*y}$
for all $x\in X$ and $y\in Y$. Since $T^*y \in Y$ by assumption, it follows that the map $x \mapsto
\dual{Tx}{y}$ is $\sigma$-continuous and thus, since $y$ was arbitrary, $T \in L(X,\sigma )$.
Finally, we have
\[\norm{T'}_{L(Y)} = \sup_y\sup_x |\dual{x}{T'y}| =
\sup_x\sup_y|\dual{Tx}{y}| = \norm{T}_{L(X)} \,\, , \]
where all suprema are taken over elements of norm at most 1.

(ii) Let $(T_n) \subset L(X,\sigma )$ be given with $T_n\to T \in L(X)$ in the operator norm. 
By (i), it suffices to prove $T^*Y\subset Y$. Let $y\in Y$ be given. By assumption, $T_n'y \in Y$. Furthermore,
\[ \norm{T'_ny-T'_my} = \sup_{\|x\|\leq 1}|\dual{T_nx-T_mx}{y}| \leq \norm{T_n-T_m}\cdot \norm{y} \,\, . \]
Thus $T'_ny$ is a Cauchy sequence in $Y$ and hence converges to some $\tilde{y}\in Y$. Now for
arbitrary $x\in X$ we have $\dual{Tx}{y} = \lim \dual{x}{T_n'y} = \dual{x}{\tilde{y}}$, proving
that $T^*y = \tilde{y}\in X$. This finishes the proof.
\end{proof}

In the study of transition semigroups, one is mainly interested in positive contraction operators which are kernel operators, as they give the transition probabilities for a Markov process.
Let us recall the following definition.

\begin{defn}
Let $(E,\Sigma )$ be a measurable space. A {\it bounded kernel} on $E$ is a mapping
$k : E\times \Sigma \to \CC$ such that
\begin{enumerate}
\item $k(x, \cdot )$ is a complex measure on $(E,\Sigma )$ for all $x \in E$;
\item $k(\cdot , A)$ is measurable for all $A \in \Sigma$;
\item $\sup_{x\in E} |k|(x, E) < \infty$. Here, $|k|(x, \cdot )$ is the 
total variation of $k(x, \cdot )$.
\end{enumerate}
A linear operator $T$ on a closed subspace $X$ of $B_b(E)$ is called a {\it kernel operator (on $X$)} if
there exists a bounded kernel $k$ on $E$ such that
\begin{equation}\label{eq.kernelop}
(Tf)(x) = \int_E f(y) \, k(x,dy) \quad , \, \forall\, f \in X\,\, .
\end{equation}
\end{defn} 

We now prove that for many spaces $X \subset B_b(E)$ a kernel operator on $X$ is the same as a 
$\sigma$-continuous operator for the norming dual pair $(X,\mathcal{M})$.
We need some preparation. 

If $S$ is any set of functions, we denote by $\sigma (S)$ the $\sigma$-algebra generated
by $S$, i.e. the smallest $\sigma$-algebra such that every $f\in S$ is measurable with 
respect to this $\sigma$-algebra. If $S$ is a Stonean vector lattice, i.e.\ a 
vector lattice of functions such that if $f\in S$ is real then also $\inf\{f,\one\}
\in S$, then the system
\[ \mathcal{E}(S) := \{ \, A \, : \, \exists\, u_n \in S\,\,\,\mbox{such that}\,\,\,
0\leq u_n \uparrow \one_A \,\,\mbox{pointwise}\, \} \]
generates $\sigma (S)$ and is closed under finite intersections, see
\cite[Lemma 39.4]{bauer}.

\begin{defn}
Let $(E, \Sigma )$ be a measurable space and $X \subset B_b(E)$ be a $\norm{}_{\infty}$-closed 
subspace of $B_b(E)$ which is a Stonean vector lattice. Further, let $\mathcal{M}_{(0)}(E)$ denote
either $\mathcal{M}(E)$ or $\mathcal{M}_0(E)$. In the latter case we additionally assume that $E$
is a completely regular Hausdorff space. Then $X$ is called a $\mathcal{M}_{(0)}(E)$-{\it transition
space} for $E$ if
\begin{enumerate}
\item $(X,\mathcal{M}_{(0)}(E))$ is a norming dual pair (with the duality \eqref{eq.duality});
\item $\sigma (X) = \Sigma $;
\item There exists a sequence $(f_n)_{n\in\CN} \subset X$ such that $0 \leq f_n \uparrow \one$ pointwise.
\end{enumerate}
\end{defn}

\begin{example}
For every measurable space $(E,\Sigma )$, the space $B_b(E)$ is a $\mathcal{M}(E)$-transition 
space for $E$. If $E$ is a metric space, then $C_b(E)$ is a $\mathcal{M}_0(E)$-transition space for $E$. Indeed,
$\mathcal{E}(C_b(E))$ contains every open $F_{\sigma}$-set and hence, since $E$ is a metric
space, every open set. Thus $\sigma (C_b(E)) = \mathcal{B}(E)$.
\end{example}

The following is a generalization of Theorem 4.8.1. in \cite{jacob}.

\begin{prop}\label{t.extend}
Let $(E,\Sigma )$ be a measurable space and $X$ be a $\mathcal{M}_{(0)}(E)$-transition space for $E$. 
Denote by $\sigma$ the $\sigma (X, \mathcal{M}_{(0)}(E))$-topology. Consider the following
statements:
\begin{enumerate}
\item $T\in L(X, \sigma )$;
\item $T$ is a kernel operator on $X$.
\end{enumerate}
Then (i) $\Rightarrow$ (ii). In this case, $T$ has a unique extension to a kernel operator on $B_b(E)$. 
If $\mathcal{M}_{(0)}(E) = \mathcal{M}(E)$, then also (ii) $\Rightarrow$ (i).
\end{prop}

\begin{proof}
(i) $\Rightarrow$ (ii): If $T\in L(X,\sigma)$, then $k(x,\cdot ):= T'\delta_x \in \mathcal{M}_{(0)}(E)$. 
By definition, we have $(Tf)(x) = \dual{Tf}{\delta_x} = \dual{f}{T'\delta_x} = \int f(y) k(x,dy)$.
Furthermore,
$\sup_x|k|(x,E) \leq \norm{T} < \infty$. It remains to prove that $k(\cdot , A)$ is
measurable for any $A \in \Sigma$. Denote the collection of sets $A$ for which this is
true by $\mathcal{G}$. Then $\mathcal{E}(X) \subset \mathcal{G}$. Indeed, if $A\in
\mathcal{E}(X)$, then there exists a sequence $(u_n)_{n\in\CN} \subset X$ with $0\leq u_n\uparrow \one_A$.
Now the dominated convergence theorem yields
\[ k(x,A) = \dual{\one_A}{T'\delta_x}= \lim_{n\to\infty} \dual{u_n}{T'\delta_x} = \lim_{n\to\infty} (Tu_n)(x) \,\, ,\]
for all $x \in E$.
Hence $k(\cdot, A)$ is measurable as the pointwise limit of measurable functions. 
Using the properties of a bounded kernel, it is easy to see that $\mathcal{G}$ is a Dynkin
system. Now $\mathcal{G} = \Sigma$ follows from the Dynkin $\pi\minus \lambda$ theorem since
$\mathcal{E}(X)$ is closed under finite intersections.

(ii) $\Rightarrow$ (i): By hypothesis, there exists a kernel $k$ such that \eqref{eq.kernelop}
holds for all $f \in X$. However, the right hand side of \eqref{eq.kernelop} also defines a bounded linear 
operator on $B_b(E)$ (which we still denote by $T$). We may also define an operator $S$ on $\mathcal{M}(E)$ by
\[ (S\mu )(A) := \int_E k(x,A)\, d\mu (x) \,\, . \]
It is easy to see that $S \in L(\mathcal{M}(E))$. However, for $f = \one_A$, we have
\[ \dual{Tf}{\mu} = \int_E k(x,A)\, d\mu = \dual{f}{S\mu} \quad \forall\, \mu \in \mathcal{M}(E)\,\,.\] 
Using linearity and approximation, we see that the above equation holds for arbitrary $f \in B_b(E)$. 
This proves $T^*\mathcal{M} \subset \mathcal{M}$ and hence (i) by Proposition \ref{p.operator}.
\end{proof}

\section{A Variant of the Pettis Integral}\label{sect3}
Throughout this section we fix a norming dual pair $(X,Y)$ and 
a $\sigma$-finite measure space $(\Omega, \mathscr{F}, m )$. 

\begin{defn}
A function $f: \Omega \to X$ is called {\it scalarly} $Y$-{\it measurable} 
({\it scalarly} $Y$-{\it integrable}), if the function $\omega \mapsto \langle f(\omega ),y\rangle$ is 
measurable (integrable) for every $y\in Y$.
\end{defn}

As in the proof of Lemma 1 in Section II.3 of \cite{diesteluhl}, one sees that
if $f$ is scalarly $Y$-integrable, then for any $A \in \cF$ the linear functional
$\varphi_A := [ y \mapsto \int_A \langle f(\omega ),y\rangle \, dm]$
is norm continuous and hence an element of $Y^*$.

\begin{defn}
If $f$ is scalarly $Y$-integrable, then the element $\varphi_A$ of $Y^*$ is called
{\it the $Y$-integral of $f$ over} A. We write $\int_A f\, dm := \varphi_A$.
If $\varphi_A \in X \subset Y^*$, for every $A \in \cF$, we say that $f$ is {\it $Y$-integrable}. 
\end{defn}

If $f$ is $Y$-integrable, then, by definition, we may interchange integration and application
of linear functionals in $Y$. The following lemma shows that the same is true for
linear operators in $L(X,\sigma )$. We omit its easy proof.

\begin{lem}\label{l.interchange}
Let $f: \Omega \to X$ be scalarly $Y$-integrable such that $\int_{\Omega}f\, dm \in X$. 
Then, for $T \in L(X,\sigma )$, 
the function $Tf$ is scalarly $Y$-integrable and we have $\int_\Omega Tf\,dm = T\int_\Omega f\, dm \in X$.
\end{lem}

Our main result about $Y$-integrability is the following.

\begin{thm}\label{t.criteria}
Assume that there exists a consistent topology $\tau$ on $X$ 
such that $(X, \tau )$ is quasi-complete, i.e. $\tau$ is complete on every bounded, closed
subset of $(X, \tau )$.
Then every almost $\tau$-separably valued, scalarly $Y$-integrable function $f: \Omega \to X$,
such that $\|f\|$ is majorized by an integrable function, is $Y$-integrable.
Here, $f$ is called almost $\tau$-separably valued if there exists a null set $N$ and
a $\tau$-separable subspace $X_0$ of $X$ such that $f(\Omega \setminus N) \subset X_0$.
\end{thm}

\begin{rem}
As a consequence of \cite[\S 18 4.(4)]{koethe}, there exists a quasi-complete consistent topology
$\tau$ on $X$ if and only if $\mu (X,Y)$ is quasi-complete.  
\end{rem}
We first prove some preliminary lemmata which will also be used independently of the theorem.

\begin{lem}\label{l.yint}
Assume that $f: \Omega \to X$ is scalarly $Y$-measurable and that $\|f\|$ is majorized by an integrable function $g$.
Then $f$ is scalarly $Y$-integrable and the $Y$-integral of $f$ over any $A\in \cF$ is sequentially $\sigma'$-continuous and satisfies the estimate
\begin{equation}\label{eq.integralestimate}
\left\|\int_A f\, dm\right\|_{Y^*} \leq \int_{A} g(\omega ) \, dm(\omega) \,\, .
\end{equation}
\end{lem}

\begin{proof}
As $f$ is scalarly $Y$-measurable and satisfies the estimate $|\langle f(\cdot ),y\rangle | \leq 
g(\cdot )\|y\|$, it follows that $f$ is scalarly $Y$-integrable. 
Integrating this inequality and taking the supremum over $y$ with 
$\norm{y}\leq 1$, estimate \eqref{eq.integralestimate} follows.
Now, let $(y_n)_{n\in\CN}$ be a sequence in $Y$ and assume
$y_n \weak' y \in Y$.  Then $\langle f,y_n\rangle \to \langle f,y\rangle$
pointwise on $\Omega$. By Proposition \ref{p.bounded}, $\|y_n\|$ is bounded,
say by $M$.  Hence $|\langle f , y_n \rangle| \leq M \cdot g$.
Thus $\varphi_A$ is sequentially $\sigma'$-continuous by the dominated convergence theorem.
\end{proof} 

\begin{lem}\label{l.dom}
Let $f : \Omega \to X$ be a scalarly $Y$-measurable function such that $\norm{f}\leq g$ a.e.\ for some
integrable function $g$. Furthermore, let $(\alpha_n)_{n\in\CN}$ be a bounded sequence in $L^{\infty}(m)$ converging
pointwise a.e.\ to $\alpha \in L^{\infty}(m)$. Then $\int_{\Omega} \alpha_n f \, dm$
converges to $\int_{\Omega} \alpha f\, dm$ with respect to the norm in $Y^*$.
In particular, $A \mapsto \int_A f\, dm$ defines a countably additive vector measure 
with values in $Y^*$. 
\end{lem}

\begin{proof}
By \eqref{eq.integralestimate}, we have
\[ \left\| \int_{\Omega} \alpha_n f\, dm - \int_{\Omega} \alpha f\, dm\right\|_{Y^*}
 \leq \int_\Omega |\alpha_n - \alpha | g\, dm \to 0\,\, ,
\]
by dominated convergence. The addendum follows by applying this to 
$\alpha_n := \one_{\bigcup_1^nA_k}$ and $\alpha := \one_{\bigcup_1^\infty A_k}$
for some sequence $(A_k)_{k\in\CN} \subset \cF$ of pairwise disjoint sets.
\end{proof}

\begin{proof}[Proof of Theorem \ref{t.criteria}]
We first make some simplifying assumptions.

We assume without loss of generality that the set $N$ is empty, 
otherwise changing $f$ on a set of measure 0. We may furthermore assume that $(X, \tau )$ is separable. 
If this is not the case, we replace $X$ by $X_1:=\overline{X_0}^{\,\tau}$ and $Y$ by $Y/X_1^\perp$. 
Since the norm topology is finer than $\tau$, the space $X_1$ is norm closed in $X$ and hence a Banach space. 
Furthermore, $(X_1,\tau |_{X_1})$ is a quasi-complete locally convex space and, 
as a consequence of the Hahn-Banach theorem, we have $(X_1, \tau |_{X_1} )' = Y/X_1^{\perp}$.

Last, we may assume that $\norm{f}$ is bounded. Indeed, assuming that $\norm{f} \leq g \in L^1(m)$, we may
consider $f_n := \one_{A_n}f$, where $A_n := \{ g\leq n \} \in \cF$. If we know that the $Y$-integral of $f_n$
over some set $A$ belongs to $X$ for every $n \in \CN$, then so does the $Y$-integral of $f$ over
the set $A$ by Lemma \ref{l.dom} and the closedness of $X$ in $Y^*$.\medskip

Denote the completion of $(X,\tau )$ by $(\tilde{X}, \tilde{\tau})$. Then
$(\tilde{X}, \tilde{\tau})$ is locally convex and separable. Furthermore, by \cite[\S 21.4 (5)]{koethe},
$(\tilde{X}, \tilde{\tau} )' = Y$. 

Now let $A \in \cF$ with (strictly) positive finite measure be given. 
By Lemma \ref{l.yint}, the $Y$-integral $\varphi_A \in Y^*$ of $f$ over
$A$ is sequentially $\sigma'$-continuous and hence in particular sequentially $\sigma (Y, \tilde{X})$-continuous.
Since $(\tilde{X}, \tilde{\tau } )$ is complete and separable, $\varphi_A$ is $\sigma (Y, \tilde{X})$-continuous
by \cite[\S 21.9 (5)]{koethe} and thus $\varphi_A \in (Y, \sigma (Y, \tilde{X}))' = \tilde{X}$.

Now consider $B_0 = \mathrm{co}\{f(\omega )\, : \, \omega \in \Omega\}$. Then $B_0$ is convex and bounded
and hence so is its $\tilde{\tau}$-closure $B$. Since $(X, \tau )$ is quasi-complete, $B \subset X$.

We claim that $m(A)^{-1} \varphi_A \in B$. Indeed, if this was not the case, then, by the Hahn-Banach 
theorem, there exist $\eps > 0$ and $y \in Y = (\tilde{X}, \tilde{\tau})'$ such that
$\Re \dual{f(\omega )}{y} + \eps \leq m(A)^{-1}\Re\dual{\varphi_A}{y}$ for every $\omega \in \Omega$. 
Integrating this equation
yields $\Re \dual{\varphi_A}{y} + \eps m(A) \leq \Re \dual{\varphi_A}{y}$ -- a contradiction since
$m(A) > 0$. It follows that $\varphi_A \in X$. 

For a general set $A$ of positive measure, 
approximate $A$ by a sequence $(A_n)_{n\in\CN}$ with $0<m(A_n)< \infty$ and use Lemma \ref{l.dom}.
\end{proof}

Let us briefly discuss the assumptions of Theorem \ref{t.criteria}.

{\it The case $Y = X^*$ --} A function $f: \Omega \to X$ is $X^*$-integrable iff it is Pettis integrable in
the classical sense. Note that the norm topology on $X$ is a complete, consistent topology. 
Furthermore the assumption that $f$ is scalarly $X^*$-measurable and almost $\norm{}$-separably 
valued imply that $f$ is strongly measurable
by the Pettis measurability theorem \cite[II.1, Theorem 2]{diesteluhl}. Thus in this case, if $f$ satisfies
the hypothesis of Theorem \ref{t.criteria}, then $f$ is Bochner integrable.

In Pettis measurability theorem the assumption of scalar $X^*$-measurability can actually be weakened to scalar $Y$-measurability for any norming subset $Y \subset X^*$, see Corollary 4 in II.1 of \cite{diesteluhl}. 
We note that since we only require the range of $f$ to be almost $\tau$-separable in Theorem \ref{t.criteria}, 
in the case of an arbitrary $Y$ we do not implicitly require that $f$ is strongly measurable.\medskip 

{\it The case $X=Y^*$ --} In this case, the $Y$-integral coincides with the $\mathrm{weak}^*$-integral.
Hence every scalarly $Y$-integrable function $f: \Omega \to Y^*$ is 
$Y$-integrable. We note that since closed, bounded balls in $Y^*$ are $\mathrm{weak}^*$-compact,
the $\mathrm{weak}^*$-topology is quasi-complete. We also note that in this case 
the separability assumption in Theorem \ref{t.criteria} is not needed.\medskip

The above examples are extensively studied in the literature. The following is our
basic example of a norming dual pair on which a complete, consistent topology exists.

\begin{example}\label{ex.cb}
 Let $E$ be a completely regular Hausdorff space and consider the norming dual pair
$(C_b(E),\mathcal{M}_0(E))$. Then, by Section 7.6 of \cite{jarchow}, the {\it strict topology}
is a consistent topology on $X$. It is complete if and only if $C(E)$, the space of {\emph all}
continuous functions on $E$, is complete with respect to the compact-open topology, see Section 3.6 of
\cite{jarchow}. If $E$ is metrizable or locally compact, this is certainly the case.
\end{example}

The question arises whether on {\it every} norming dual pair there exists a quasi-complete, consistent topology.
This question was answered to the negative by {\sc Bonet} and {\sc Cascales} \cite{boca}. 
In Section \ref{sect5} we will give a concrete example that the assertion of Theorem \ref{t.criteria}
may fail without the assumption that there exists a quasi-complete consistent topology.\medskip

The following result is useful in establishing $Y$-integrability. 

\begin{prop}\label{p.construct}
Let $\mathcal{E}$ be a generator of $\cF$ which is closed under finite intersections and 
$f: \Omega \to X$ be a scalarly $Y$-measurable function with the following properties:
\begin{enumerate}
\item There exists a measurable function $g$ such that $\norm{f}\leq g$;
\item There exists a sequence $(\Omega_n)_{n\in\CN} \subset \cF$ with $m(\Omega_n)<\infty$ for all $n \in \CN$ 
and $\bigcup_{n\in\CN} \Omega_n = \Omega$ such that the function $g$ from (i) satisfies $g\one_{\Omega_n}\in L^1(m)$ for all $n\in \CN$;
\item $x_A := \int_A f \, dm \in X$ for every $A \in \mathcal{E}\cup\{\Omega_n\, : \, n \in \CN \}$.
\end{enumerate}
Then, for every measurable function $\alpha : \Omega \to \CC$ with $|\alpha|g\in L^1(m)$,
the function $\alpha f$ is $Y$-integrable.
\end{prop}

\begin{proof} 
By Lemma \ref{l.yint}, $\alpha f$ is scalarly $Y$-integrable on $\Omega$. It suffices
to prove that its $Y$-integral over $\Omega$ belongs to $X$, as we can clearly replace $\alpha$
by $\alpha\cdot\one_A$ for any $A \in \cF$. 
We proceed in three steps.

{\it Step 1 --} 
Let $n\in \CN$ be arbitrary and let $\mathcal{D}_n$ denote the collection of
all sets $A \in \cF$ such that $\int_{A\cap\Omega_n}f\, dm \in X$.
By assumption (iii), $\mathcal{E} \subset \mathcal{D}_n$. Using Lemma \ref{l.dom}, it is easy to 
see that $\mathcal{D}_n$ is a Dynkin system. Hence $\mathcal{D}_n = \cF$ by Dynkin's $\pi$-$\lambda$ theorem. 

{\it Step 2 --} Now we prove the assertion for a simple function $\alpha$. 
By Step 1 and linearity, the $Y$-integral of $\alpha f$ over $\Omega_n$ is an element of $X$. 
By Lemma \ref{l.dom}, $\int_{\Omega_n}\alpha f\,dm \to \int_{\Omega}\alpha f\, dm$
in $Y^*$, hence $\int_{\Omega}\alpha f\, dm \in X$. 

{\it Step 3 --} Now let $\alpha$ be arbitrary. Then there exists a sequence of step functions
$(\alpha_k)_{k\in\CN}$ such that $|\alpha_k| \leq |\alpha|$ and $\alpha_k\to \alpha$ pointwise. By Step
2, $\int_{\Omega}\alpha_k f\, dm \in X$ for every $k$. Again by Lemma \ref{l.dom} 
it follows that $\int_{\Omega} \alpha f\, dm \in X$.
\end{proof}

\section{Semigroups and Their Laplace Transforms}\label{sect4}

\begin{defn}
Let $(X,Y)$ be a norming dual pair. A {\it semigroup}
on $(X,Y)$ is a family of operators $\mathbf{T} = (T(t))_{t\geq 0} \subset L(X, \sigma )$ 
such that $T(t+s) = T(t)T(s)$ for all $t,s\geq 0$ and $T(0) = id_X$. A semigroup
is called {\it exponentially bounded}, if there exist
$M\geq 1$ and $\omega \in \CR$ such that $\norm{T(t)}\leq Me^{\omega t}$.
In this case, we say that $\mathbf{T}$ is of {\it type} $(M,\omega )$. 
A semigroup of some type $(M,\omega )$ is called {\it integrable} if
$t\mapsto \dual{T(t)x}{y}$ is measurable for all $x \in X$ and $y \in Y$ and there
exists a complex number $\lambda_0$ with $\Re \lambda_0>\omega$ and an operator $R_0 \in L(X,\sigma )$
such that
\begin{equation}\label{eq.resolvent}
 \dual{R_0x}{y} = \int_0^{\infty} e^{-\lambda_0 t}\dual{T(t)x}{y}\, dt 
\quad , \quad \forall\, x \in X\, , \, y \in Y \,\, .
\end{equation}
\end{defn}

Two remarks are in order. Let us first note that for a fixed $\lambda_0$ there is at most
one operator $R_0 \in L(X, \sigma )$ such that \eqref{eq.resolvent} is satisfied. Second, note
that the definition of `integrable semigroup' is symmetric, i.e.\ if
$\mathbf{T}$ is an integrable semigroup on $(X,Y)$, then the $\sigma$-adjoint semigroup
$\mathbf{T}'$ is an integrable semigroup on $(Y,X)$. To see this note that if $R_0 \in L(X, \sigma )$,
then $R_0' \in L(Y, \sigma')$. Furthermore, we have $\dual{x}{R_0'y} = \int_0^\infty e^{-\lambda_0t}
\dual{x}{T(t)'y}\, dt$ for all $x\in X$ and $y \in Y$.\medskip

We will see in a moment that if $\mathbf{T}$ is an integrable semigroup, then for {\it every}
$\lambda$ with $\Re\lambda > \omega$ there exists an operator $R(\lambda ) \in L(X, \sigma )$
such that
\begin{equation}\label{eq.resolvent2}
 \dual{R(\lambda )x}{y}  = \int_0^\infty e^{-\lambda t}\dual{T(t)x}{y}\, dt \quad , \quad \forall\, x \in X\,,\,
y\in Y\,\, .
\end{equation}
Clearly, $R(\lambda_0 ) = R_0$. The family $\mathbf{R} := (R(\lambda ))_{\Re\lambda >
\omega}$ is called the {\it Laplace transform} of $\mathbf{T}$. 

It is well known that the Laplace transform of a strongly continuous semigroup is the resolvent of its
generator. Since we did not impose continuity assumptions, we cannot expect the Laplace
transform to be injective. In particular, it may not be the resolvent of an operator. However,
the following proposition shows that, similar as in \cite{are01}, the Laplace transform
of an integrable semigroup is a pseudo-resolvent. 

We will use freely some results about pseudo-resolvents and multivalued (m.v.\ for short) operators.
We refer the reader to \cite{are01} or Appendix A of \cite{haase} for more information.

\begin{prop}\label{p.pseudo}
Let $\mathbf{T}$ be an integrable semigroup of type $(M,\omega )$. 
Then there exists a pseudo-resolvent $(R(\lambda ))_{\Re\lambda >
\omega} \subset L(X, \sigma )$ 
such that \eqref{eq.resolvent2} holds for every $\Re\lambda > \omega$. Furthermore, every
$R(\lambda )$ commutes with every $T(t)$ and for $\Re\lambda > \omega$ and $k \in \CN$ we have 
$\norm{(\Re\lambda - \omega )^kR(\lambda )^k}\leq M$. 
\end{prop}

\begin{proof}
By the definition of `integrable semigroup' there exists some $\lambda_0 \in \{\Re\lambda > \omega\}$
and $R_0 \in L(X,\sigma )$ such that \eqref{eq.resolvent} holds. 
Define the m.v.\ operator $\mathcal{A}$ by $\mathcal{A} := \lambda_0 - R_0^{-1}$ and put
$R(\lambda ) := (\lambda - \mathcal{A})^{-1}$. Now define
\[ \Omega_0 := \{ \lambda \, : \Re\lambda >\omega\, ,\, R(\lambda ) \in L(X,\sigma )\,\,\mbox{and
 \eqref{eq.resolvent2} holds}\,\}\,\, .
\]
Then we have $\lambda_0 \in \Omega_0 \subset \Omega := \{\lambda\, :\, \Re\lambda > \omega\,,\, R(\lambda )
\in L(X)\, \}$. By \cite[Proposition A.2.3]{haase}, the $L(X)$-valued map $R : \Omega \to L(X)$ defines a pseudo-resolvent; 
in particular, $\Omega$ is open and  $R$ is holomorphic. 
More precisely, if $\lambda \in \Omega$ and $|\lambda - \mu | < \norm{R(\lambda )}^{-1}$, then $\mu \in 
\Omega$ and
\begin{equation}\label{eq.power}
 R(\mu ) = \sum_{k=0}^{\infty}(\lambda - \mu)^kR(\lambda )^{k+1} \,\, .
\end{equation}
Now fix $\lambda \in \Omega_0$ and $\mu \in B(\lambda , \norm{R(\lambda )}^{-1})$. 
Equation \eqref{eq.power} and Proposition \ref{p.operator} (ii) imply 
that $R(\mu ) \in L(X,\sigma )$. Now note that for any $\nu \in \Omega_0, \,x \in X$
and $y\in Y$ we have
\begin{eqnarray}
 \dual{R(\nu )^kx}{y} & = &
\int_{(0,\infty)^k} e^{-\nu (t_1+\cdots + t_k)}\dual{T(t_1+\cdots + t_k)x}{y}\,
dt_1\cdots dt_k\nonumber\\
& = & \int_0^\infty \frac{t^{k-1}}{(k-1)!}e^{-\nu t}\dual{T(t)x}{y}\, dt \label{eq.rk}\,\, .
\end{eqnarray}
Here the first equality follows from the semigroup law and the second equality is derived
from that fact that the $k$-fold convolution of exponential distributions is a gamma distribution.
Thus since $|\lambda -\mu|< \norm{R(\lambda )}^{-1}$, we have
\begin{eqnarray*}
\dual{R(\mu )x}{y} & = & \sum_{k=0}^\infty \dual{(\lambda - \mu )^k R(\lambda )^{k+1}x}{y}\\
& = & \int_0^\infty \sum_{k=0}^\infty \frac{((\lambda - \mu )t)^k}{k!}e^{-\lambda t}\dual{T(t)x}{y}\, dt\\
& = & \int_0^\infty e^{-\mu t}\dual{T(t)x}{y}\,dt \,\, , 
\end{eqnarray*}
for all $x\in X, y \in Y$. Hence $\mu \in \Omega_0$. Since $\lambda \in \Omega_0$ was arbitrary, it follows
that $\Omega_0$ is an open subset of $\Omega$.

Now assume that $(\lambda_n)_{n\in\CN}$ is a sequence in $\Omega_0$ converging to some $\lambda \in \Omega$.
Then $R(\lambda_n) \to R(\lambda )$ in the operator norm and hence $R(\lambda ) \in 
L(X,\sigma )$ by Proposition \ref{p.operator} (ii). 
Fix $\gamma > \omega$ such that $\Re\lambda_n > \gamma$ for all $n \in \CN$. Using the estimate
$|e^{-\lambda_n t}\dual{T(t)x}{y}| \leq Me^{(\omega - \gamma )t}\norm{x}\cdot \norm{y} \in L^1(0,\infty)$
for all $x\in X$ and $y \in Y$, we may infer from dominated convergence that $\dual{R(\lambda )x}{y}
= \int_0^\infty e^{-\lambda t}\dual{T(t)x}{y}\, dt$ for all $x \in X,y\in Y$. This proves
that $\Omega_0$ is closed in $\Omega$. It follows that $\Omega_0$ contains the connected
component of $\lambda_0$ in $\Omega$.

Let us prove now that $\Omega_0 = \{\lambda : \Re\lambda > \omega \}$. To that end, let
$(\lambda_n)_{n\in\CN}$ be a sequence in the connected component of $\lambda_0$ in $\Omega$ converging
to some $\lambda$ in the boundary of that component. By \cite[Proposition 3.5]{are01}, $\norm{R(\lambda_n)}$ 
must be unbounded.
If this is the case, we infer from the uniform boundedness principle that
we find some $x \in X$ and some $y \in Y$ such that $\dual{R(\lambda_n)x}{y}$ is unbounded. But this is
impossible unless $\Re\lambda = \omega$. Indeed, if $\Re\lambda > \omega$, then, similar as above, we find
\[ \limsup_{n\to\infty} |\dual{R(\lambda_n )x}{y}| \leq \limsup_{n\to\infty} M
\norm{x}
 \cdot \norm{y} \int_0^\infty e^{(\omega - \Re\lambda_n)t} \, dt < \infty \,\, .
\]
Hence we must have $\Re\lambda = \omega$ and, thus, $\Omega_0 = \{\lambda : \Re\lambda > \omega\}$. 

The fact that every $R(\lambda )$ commutes with every $T(t)$ is an easy consequence of Lemma
\ref{l.interchange} and the semigroup law. The estimate $\norm{(\Re\lambda - \omega )^kR(\lambda )^k}\leq M$
may be deduced from \eqref{eq.rk} and the exponential boundedness of $\mathbf{T}$.
\end{proof}

The question arises whether an integrable semigroup is uniquely determined by its Laplace transform. Without
further assumptions, this is not the case, not even if $Y=X^*$, see \cite{phillips}. We need the following

\begin{defn}\label{d.countably}
Let $X$ be a Banach space and $M$ be a subspace of $X$. A subset $W \subset X^*$ is said 
to {\it separate points in $M$} if for every $x\in M\setminus \{0\}$ there
exists $w\in W$ with $\dual{x}{w}\neq 0$. A norming dual pair $(X,Y)$ is 
said to be {\it countably separated}\index{countably separated} 
if there exists a countable subset of $X$ separating points in $Y$ and 
there exists a countable subset of $Y$ separating points in $X$.
\end{defn}

\begin{thm}\label{t.laplaceunique}
Let $\mathbf{T},\mathbf{S}$ be integrable semigroups on $(X,Y)$ of type $(M_T,\omega_T)$
and $(M_S, \omega_S)$ respectively. Suppose that for the corresponding Laplace transforms we have
$R_{\mathbf{T}}(\lambda ) = (\lambda - \mathcal{A})^{-1} = R_{\mathbf{S}}(\lambda )$ 
for some $\lambda> \max\{\omega_\mathbf{T}, \omega_{\mathbf{S}}\}$. 
Then $T(t) = S(t)$ for all $t\geq 0$, provided one of the following conditions is satisfied:
\begin{enumerate}
\item $D(\mathcal{A})$ is $\sigma$-dense in $X$; 
\item $(X,Y)$ is countably separated.
\end{enumerate}
\end{thm}

The proof uses the following lemma which is taken from \cite[Lemma 3.16.5]{abhn}.

\begin{lem}\label{l.nullset}
Let $M \subset (0, \infty )$ be a set of Lebesgue measure 0 and assume that
$t,s \not\in M$ implies $t+s\not\in M$. Then $M = \emptyset$.
\end{lem}

\begin{proof}[Proof of Theorem \ref{t.laplaceunique}]
As a consequence of Proposition \ref{p.pseudo}, we have $(\lambda - \mathcal{A})^{-1} = R_\mathbf{T}(\lambda ) = 
R_\mathbf{S}(\lambda ) \in L(X, \sigma )$ for all $\lambda > \max\{\omega_T, \omega_S\}$.
Hence, for such $\lambda$ and any $x\in X$ and $y \in Y$ we have
\[ \int_0^{\infty} e^{-\lambda t}\dual{T(t)x}{y} \, dt
= \int_0^{\infty} e^{-\lambda t}\dual{S(t)x}{y} \, dt \,\, .\]

By the uniqueness theorem for Laplace transforms \cite[Theorem 1.7.3]{abhn}, there
exists a set $N(x,y)$ of Lebesgue measure zero such that $\dual{T(t)x}{y} = \dual{S(t)x}{y}$ 
for all $t\not\in N(x,y)$.

First assume (i). Note that for every $\Re\lambda > \omega\, ,\,u \in X$ and $y \in Y$ we have
\begin{eqnarray*}
\dual{T(t)R_{\mathbf{T}}(\lambda )u}{y} & = & \int_0^\infty e^{-\lambda s} \dual{T(t+s)u}{y}\, ds = e^{\lambda t}
\int_t^\infty e^{-\lambda r}\dual{T(r)u}{y}\, dr\\
& = & e^{\lambda t}\left(\dual{R_{\mathbf{T}} (\lambda )u}{y} - \int_0^te^{-\lambda r}\dual{T(r)u}{y}\, dr \right) \,\, , 
\end{eqnarray*}
and thus
\begin{equation}\label{eq.tint} 
 \int_0^te^{-\lambda r}\dual{T(r)u}{y}\, dr = \dual{R_\mathbf{T} (\lambda )u}{y} - e^{-\lambda t}
\dual{T(t)R_{\mathbf{T}}(\lambda )u}{y}
 \,\, .
\end{equation}
Now let $x \in D(\mathcal{A}) = \rg R_{\mathbf{T}}(\lambda )$, say
$x = R_{\mathbf{T}}(\lambda )z$. Then the above equation for $u = z$ and arbitrary $y \in Y$ yields
\[ \dual{T(t)x}{y} = e^{\lambda t}\left(\dual{x}{y} - 
\int_0^t e^{-\lambda r} \dual{T(r)z}{y} \, dr \right) \,\, ,\]
implying that $t \mapsto \dual{T(t)x}{y}$ is continuous.
The same applies to the corresponding orbit of $\mathbf{S}$ and we find $N(x,y) = \emptyset$. Thus 
$T(t)x = S(t)x$ for every $t \geq 0$ and $x \in D(\mathcal{A})$. 
However, if the $\sigma$-continuous linear operators $T(t)$ and $S(t)$ coincide on
the $\sigma$-dense subspace $D(\mathcal{A})$, then they are equal.

Now assume that (ii) is satisfied. 
Let $\{x_n\}_{n\in\CN} \subset X$ and $\{y_n\}_{n\in\CN} \subset Y$
be countable subsets separating points in $Y$ and $X$ respectively.
Fix $x \in X$ and put $N(x) = \bigcup_{n\in \CN}N(x,y_n)$. 
Then $N(x)$ is a null set and
\[ \dual{T(t)x}{y_n} = \dual{S(t)x}{y_n} \quad \forall \,t \not\in N(x)
\,\,,\,\, n\in \CN \,\, .\]
Since $\{y_n\}$ separates points, $T(t)x = S(t)x$ for all
$t\not\in N(x)$. In particular, $\dual{T(t)x}{y} = 
\dual{S(t)x}{y}$ for all $t\not\in N(x)$ and all $y\in Y$.

Now fix $y\in Y$ and put $N = \bigcup_{n\in \CN} N(x_n)$.
Then $N$ has measure 0 and for $t \not \in N$ and $n\in\CN$ we have
\[ \dual{x_n}{T(t)'y} = \dual{T(t)x_n}{y} = \dual{S(t)x_n}{y}
= \dual{x_n}{S(t)'y} \,\, .\]
As $\{x_n\}$ separates points, it follows that $T(t)'y = S(t)'y$ for all $t\not \in N$.
Since $y$ was arbitrary, $T(t) = S(t)$ for all $t \not\in N$. 
Now let $M = \{ t \, : \, T(t) \neq S(t) \}$. Then $M \subset N$, showing that $M$ has measure $0$. 
However, if $t,s\not\in M$ then, by the semigroup law,
$t+s\not\in M$. Thus Lemma \ref{l.nullset} implies $M = \emptyset$.
\end{proof}

\begin{rem}
 It is proved in \cite[Theorem 2.10]{k09b} that if $\mathbf{T}$ is {\it $\sigma$-continuous at 0},
i.e.\ $T(t)x \weak x$ as $t\downarrow 0$ for every $x \in X$, then $\rg R(\lambda ) = D(\mathcal{A})$
is $\sigma$-dense in $X$. Hence very mild continuity assumptions ensure that condition (i) in
Theorem \ref{t.laplaceunique} is satisfied.
\end{rem}

We now generalize a result from the theory of strongly
continuous semigroups, cf. \cite[Proposition 3.1.9]{abhn}. Note that in our situation 
the operator $\mathcal{A}$ may be multivalued.

\begin{prop}\label{p.rangecharacterisation}
Let $\mathbf{T}$ be an integrable semigroup on $(X,Y)$ with Laplace transform $\mathbf{R}$
and let $\mathcal{A}$ be the unique
m.v.\ operator such that $R(\lambda ) = (\lambda - \mathcal{A})^{-1}$.
\begin{enumerate}
\item The following are equivalent.
\begin{enumerate}
\item $x \in D(\mathcal{A})$ and $ z \in \mathcal{A}x $;
\item For every $t > 0$ we have $\int_0^t T(s)z \, ds = T(t)x - x$.
\end{enumerate}
\item We have $\int_0^t\, T(s)x\, ds \in D(\mathcal{A})$ and
$T(t)x - x \in \mathcal{A}\int_0^t T(s)x\, ds$ for every $x \in X$ and $t>0$.
\end{enumerate}
\end{prop}

\begin{proof}
We first note that (i) (a) is equivalent to $x = R(\lambda )(\lambda x -z)$.

(i) (a) $\Rightarrow$ (b): 
Fix $t>0$ and $y \in Y$ and define the analytic functions $f,g : \CC \to \CC$ by
\begin{eqnarray*}
f(\lambda ) &:=& \lambda \int_0^t e^{-\lambda s}\dual{T(s)x}{y} \, ds -
\int_0^te^{-\lambda s}\dual{T(s)z}{y} \, ds\\
g(\lambda ) & := & \dual{x}{y} - e^{-\lambda t}\dual{T(t)x}{y} \,\, . 
\end{eqnarray*}
Setting $u = x = R(\lambda) (\lambda x -z)$ in \eqref{eq.tint}, it follow that $f(\lambda ) = g(\lambda )$ for
all $\Re\lambda > \omega$. The uniqueness theorem for analytic functions yields $f(0) = g(0)$. 
As $t$ and $y$ were arbitrary, (b) is proved.

(b) $\Rightarrow$ (a): If $\int_0^t T(s)z \, ds = T(t)x - x$, then
\begin{eqnarray*}
 \lambda R(\lambda )x - x & = & \int_0^{\infty}\lambda e^{-\lambda t}\left(
T(t)x - x\right)\, dt
=  \int_0^{\infty}\lambda e^{-\lambda t} \int_0^t T(s)z\, ds \, dt\\
& = & \int_0^{\infty}\int_s^{\infty}\lambda e^{-\lambda t} T(s)z \, dt
\, ds = \int_0^{\infty}e^{-\lambda s} T(s)z \, ds = R(\lambda )z \,\, .
\end{eqnarray*}
It follows that $x = R(\lambda )(\lambda x -z)$. 

(ii) Considering integrals as elements of $Y^*$ at first, we have
\begin{eqnarray*}
\int_0^tT(s)x \, ds & = &
\int_0^t T(s)(\lambda - \mathcal{A})R(\lambda )x \, ds\\
& = & \lambda \int_0^t T(s)R(\lambda )x \, ds - \int_0^t
T(s)\mathcal{A}R(\lambda )x \, ds\\
& = & \lambda \int_0^tT(s)R(\lambda )x \, ds
+ R(\lambda )x - T(t)R(\lambda )x \,\, ,
\end{eqnarray*}
where we have used $R(\lambda) x \in D(\mathcal{A})$ and part (i)
in the last step. Furthermore, in slight abuse of notation, we wrote $\mathcal{A}R(\lambda )x$ in
place of an element in this set. Now note that $\int_0^tT(s)R(\lambda )x \, ds \in X$ by part (i),
hence also $\int_0^tT(s)x\, ds \in X$ by the above equation.
Now Lemma \ref{l.interchange} yields 
\[ \int_0^t T(s)x \, ds = R(\lambda )\left( \lambda \int_0^t T(s)x \, ds +x - T(t)x 
\right) \, , \]
which is equivalent to (ii).
\end{proof}

\begin{thm}\label{t.measurabledescription}
Let $\mathbf{T}$ be a semigroup of type $(M,\omega )$ on the norming dual pair $(X,Y)$.
The following are equivalent:
\begin{enumerate}
\item $\mathbf{T}$ is an integrable semigroup;
\item For every $x\in X$ the orbit $T(\cdot )x$ is locally $Y$-integrable and for every $y \in Y$ 
the orbit $T(\cdot )'y$ is locally $X$-integrable. Here `local $X/Y$-integrability' means 
$X/Y$-integrability on every bounded interval in $\CR^+$.
\end{enumerate}
\end{thm}

\begin{proof}
(i) $\Rightarrow$ (ii): As a consequence of Proposition \ref{p.rangecharacterisation} (ii), $\int_a^b T(t)x \, dt \in X$ 
for every $x \in X$ and $0\leq a < b < \infty$. As such intervals generate the Borel $\sigma$-algebra on $(0,\infty)$ 
and are closed under finite intersections, it follows from Proposition \ref{p.construct} that $T(\cdot )x$ is locally 
$Y$-integrable. Applying the same arguments to $T(\cdot )'y$ for every $y \in Y$, (ii) follows.

(ii) $\Rightarrow$ (i): Fix $\lambda$ with $\Re\lambda > \omega$. It follows from (ii) and Proposition \ref{p.construct} 
that there exists an element $R(\lambda )x \in X$ such that
$ R(\lambda )x = \int_0^{\infty} e^{-\lambda t}T(t)x\, dt$. 
It remains to prove that $R(\lambda ) \in L(X, \sigma )$. It is easy to see that $R(\lambda )$ is linear. 
Furthermore, using the exponential boundedness of $\mathbf{T}$ and the dominated convergence theorem, 
it follows that $R(\lambda ) \in L(X)$. However, arguing similar, it follows that there exists $V(\lambda ) \in L(Y)$ 
such that $V(\lambda )y = \int_0^{\infty} e^{-\lambda t} T(t)'y \,dt$.
It is easily seen that $\dual{R(\lambda )x}{y} = \dual{x}{V(\lambda )y}$, hence
$V(\lambda ) = R(\lambda )^*|_{Y}$. Proposition \ref{p.operator} implies
$R(\lambda ) \in L(X, \sigma )$. This proves (i).
\end{proof}

We end this section with the following

\begin{lem}\label{l.expb}
Let $\mathbf{T}$ be a semigroup on the norming dual pair $(X,Y)$ which is 
$\sigma$-continuous at $0$. Then $T$ is exponentially bounded.
\end{lem}

\begin{proof}
Let us first prove that $\sigma$-continuity at $0$ implies $\sup_{0\leq t\leq 1}\norm{T(t)}<
\infty$. To that end, observe that for any $x \in X$ there exists $\eps_x$ such that
$A_x := \{\,\norm{T(t)x}\, : \, 0\leq t\leq \eps_x \}$ is bounded. Indeed,
if this was wrong, there exists a sequence $t_n\downarrow 0$ such that
$\norm{T(t_n)x}$ is unbounded. However, as $T(t_n)x \weak x$, the set
$\{ T(t_n)x \}$ has to be $\sigma$-bounded and hence, by Proposition \ref{p.bounded},
norm-bounded -- a contradiction. Now the semigroup law implies that
$\{ T(t)x \,: \, 0 \leq t \leq 1\} \subset A_x\cup T(\eps_x)A_x \cup\cdots\cup
T(\eps_x)^kA_x$ for some $k\in \CN$. As all operators $T(t)$ are bounded,
it follows that $\{T(t)x\,:\, 0 \leq t\leq 1\}$ is bounded. By the uniform
boundedness principle, $\sup_{0\leq t \leq 1}\norm{T(t)} =: M< \infty$. 
Now let $\omega = \log M$. For $t\geq 0$ split $t = n+r$ for some $n \in\CN_0$ and
$r\in [0,1)$. Then $\norm{T(t)} = \norm{T(r)T(1)^n}\leq Me^{\omega n} \leq M e^{\omega t}$.
\end{proof}

\section{Integrable semigroups on $(C_b(E), \mathcal{M}_0(E))$}\label{sect5}
We now turn to the problem of integrability of transition semigroups. As we will not use  
positivity or contractivity, we will consider general semigroups of kernel operators.
Taking Theorem \ref{t.extend} into account, this is exactly the same 
as a semigroup on the norming dual pair $(B_b(E), \mathcal{M}(E))$. Our first
result states that measurability and integrability extends from $(X, \mathcal{M}_{(0)}(E))$ to
$(B_b(E), \mathcal{M}_{(0)}(E))$ if $X$ is a $\mathcal{M}_{(0)}(E)$-transition space for $E$.

\begin{lem}\label{l.integration}
Let $(\Omega, \cF, m)$ be a $\sigma$-finite measure space,
$(E,\Sigma )$ be a measurable space and let $\mathcal{M}_{(0)}(E)$ denote either $\mathcal{M}(E)$
or  $\mathcal{M}_0(E)$ (in the latter case, assume additionally that $E$ is a completely regular Hausdorff space). 
We write $\sigma$ instead of $\sigma (B_b(E), \mathcal{M}_{(0)}(E) )$. 
Let $T: \Omega \to L(B_b(E), \sigma )$ and  $X$ be a $\mathcal{M}_{(0)}$-transition space for $E$.
\begin{enumerate}
\item $T(\cdot )f$ is scalarly $\mathcal{M}_{(0)}(E)$--measurable for every $f \in B_b(E)$
if and only if $T(\cdot )f$ is scalarly $\mathcal{M}_{(0)}(E)$-measurable for every $f \in X$.
\item Assume additionally, that $\norm{T}$ is majorized by an integrable
function. Then $T(\cdot )f$ is $\mathcal{M}_{(0)}$-integrable for every 
$f \in B_b(E)$ if and only if  $T(\cdot )f$ is $\mathcal{M}_{(0)}$-integrable for every $f \in X$. 
\end{enumerate}
\end{lem}

\begin{proof}
(i) Assume that $T(\cdot )f$ is scalarly $\mathcal{M}_{(0)}$-measurable for every $f \in X$
and define 
\[ \mathcal{G} := \{ A \in \Sigma \, : \, T(\cdot )\one_A\,\,\,
\mbox{is scalarly $\mathcal{M}_{(0)}$-measurable}\, \} \,\, .\]
If $\one_A = \sup f_n$ for some sequence $(f_n)_{n\in\CN} \subset X$, then $T(\omega )f_n
\weak T(\omega )f$ for all $\omega \in \Omega$ by the $\sigma$-continuity
of $T(\omega )$. Hence, for any $\mu \in \mathcal{M}_{(0)}(E)$, we have
$\dual{T(\cdot )\one_A}{\mu} = \lim \dual{T(\cdot )f_n}{\mu}$. This
proves that $\dual{T(\cdot )\one_A}{\mu}$ is measurable. It follows
that $\mathcal{E}(X) \subset \mathcal{G}$. It is easy to see that $\mathcal{G}$
is a Dynkin system and thus $\mathcal{G} = \Sigma$. By linearity, 
$T(\cdot )f$ is $\mathcal{M}_{(0)}$-measurable for every simple function $f$. Approximating
an arbitrary function by a sequence of simple functions and using the 
$\sigma$-continuity of the operators $T(\cdot )$ again, the assertion follows.

(ii) Scalar $\mathcal{M}_{(0)}$-measurability of $T(\cdot )f$ for all $f \in B_b(E)$
follows from (i). To prove $\mathcal{M}_{(0)}$-integrability, we proceed as in (i). Define
\[ \mathcal{G} := \{ A \in \Sigma \, : \, T(\cdot )\one_A \,\, \, \mbox{is
$\mathcal{M}_{(0)}$-integrable}\, \}\,\, .\]
If $\one_A = \sup f_n$ for a sequence $(f_n)_{n\in\CN} \subset X$, then 
it follows from Lemma \ref{l.dom} that $A \in \mathcal{G}$. Hence $\mathcal{E}(X) \subset \mathcal{G}$. 
The rest of the proof is similar as in (i).
\end{proof}

We now consider semigroups of kernel operators on $C_b(E)$.

\begin{thm}\label{t.integrablecb}
Let $E$ be a completely regular Hausdorff space and $\mathbf{T}$ be a semigroup on $(C_b(E), \mathcal{M}_0(E))$
which is $\sigma$-continuous at $0$.
\begin{enumerate}
 \item If the strict topology on $C_b(E)$ is complete (cf.\ Example \ref{ex.cb}) then, for every $f\in C_b(E)$,
the orbit $T(\cdot )f$ is locally $\mathcal{M}_0$-integrable.
 \item If $E$ is a complete metric space, then $\mathbf{T}$ is integrable.
\end{enumerate}
\end{thm}

\begin{proof}
By Lemma \ref{l.expb}, the semigroup $\mathbf{T}$ is exponentially bounded, say
of type $(M,\omega )$. Furthermore, since every operator $T(t)$ is $\sigma$-continuous, the semigroup law and 
the $\sigma$-continuity at 0 imply that $t\mapsto \dual{T(t)f}{\mu}$ is right continuous
for every $\mu \in \mathcal{M}_0$ and $f \in C_b(E)$. In particular, for every $f\in C_b(E)$ 
the orbit $T(\cdot )f$ is $\mathcal{M}_0(E)$-measurable and the range of this function is $\sigma$-separable 
and hence, as a consequence of the Hahn-Banach theorem, separable with respect to any consistent topology. 
Now (i) follows from Theorem \ref{t.criteria}.\medskip

To prove (ii), note that if $E$ is a complete metric space, 
then the strict topology on $C_b(E)$ is complete, hence (i)
may be used. In view of Theorem \ref{t.measurabledescription}, 
to prove (ii) it suffices to prove that $T(\cdot )'\mu$ is locally $C_b(E)$-integrable, for every 
$\mu \in \mathcal{M}_0(E)$. Fix $\mu \in \mathcal{M}_0(E)$. Since $E$ is a metric space, 
$C_b(E)$ is a $\mathcal{M}_0$-transition space for $E$, and hence every $T(t)$ is a kernel operator by Proposition 
\ref{t.extend}. In particular, it has a unique extension to an operator $\tilde{T}(t) \in L(B_b(E), \sigma )$. 
We infer from Lemma \ref{l.integration} (i) that $t\mapsto \dual{f}{T(t)'\mu} = 
\langle \tilde{T}(t)f,\mu\rangle$ is measurable for every $f \in B_b(E)$. 
Now let $S \subset [0,\infty)$ be a bounded, measurable set.
By Lemma \ref{l.yint}, the $B_b$-integral $\varphi := \int_S T(t)'\mu \, dt$ is sequentially 
$\sigma (\mathcal{M},B_b)$-continuous. If we put $\rho (A)
= \varphi (\one_A)$, then it follows from sequential continuity that $\rho$ is a measure. 
Clearly $\varphi (f) = \int_Sf\, d\rho$ for all $f \in B_b(E)$. 

It remains to prove that
$\rho \in \mathcal{M}_0(E)$. Since $E$ is a complete metric space, a measure on $E$ is a Radon measure 
if and only if it has separable support. By assumption, the measure $T(t)'\mu$ is a
Radon measure for every $t \in S$. Consequently, we find a separable set $E_t$ such that 
$T(t)'\mu (A) = 0$ for all $A \subset E\setminus\overline{E_t}$. Define
\[ E_0 := \overline{\bigcup_{r \in S\cap \mathbb{Q}}E_r} \,\, . \]
Then $E_0$ is a separable set. We claim that $\rho$ is supported in $E_0$. Let $A \subset E\setminus
E_0$ be an open set. Then $A$ is an $F_{\sigma}$-set, say $A = \bigcup F_n$ for an increasing sequence $(F_n)_{n\in\CN}$
of closed sets. By Tieze's extension theorem, there exist functions $f_n$ such that 
$f|_{F_n} \equiv 1$ whereas $f|_{A^c}\equiv 0$. By right continuity of the paths we have
\[ \dual{f_n}{T(t)'\mu} = \lim_{r\downarrow t\, , \, r \in \mathbb{Q}}
 \dual{f_n}{T(r)'\mu} = 0 \,\, ,
\]
for all $n \in \CN$ and $t \geq 0$. Integrating over $S$ yields $\int_S \dual{f_n}{T(t)'\mu} dt = 0$. Now the
dominated convergence theorem implies that $\rho (A) = \lim_{n\to \infty} \int_S \dual{f_n}{T(t)'\mu} dt = 0$. 
This proves that $\rho$ is supported in $E_0$ and is hence a Radon measure. 
\end{proof}

\begin{rem} 
 The assumption that $\mathbf{T}$ is $\sigma$-continuous at 0 is equivalent with $\mathbf{T}$
being `stochastically continuous', cf.\ \cite[Theorem 3.8]{lanthi}
\end{rem}

\begin{example}\label{ex.sorgen}
Let $E$ denote the real line endowed with the Sorgenfrey topology $\tau_S$ which is generated by
the collection of all intervals $[a,b)$ for $a<b$. Then the shift semigroup $\mathbf{T}$, given
by $T(t)f(x) = f(x+t)$, defines a semigroup on $(C_b(E), \mathcal{M}_0(E))$ which has the following
properties. (i) it is $\sigma$-continuous at 0; (ii) for every $f \in C_b(E)$ the orbit $T(\cdot )f$
is locally $\mathcal{M}_0(E)$-integrable; (iii) $\mathbf{T}$ is not integrable. 
\end{example}

\begin{proof}
We note that $f\in C_b(E)$ if and only if $f$ is bounded and right continuous as a function on $\CR$ with
its usual topology. Furthermore, the Borel $\sigma$-algebra of $E$ is just the usual Borel $\sigma$-algebra
of $\CR$ when endowed with the usual topology. It is well known that every compact subset of $E$ is necessarily
countable (though not every countable subset of $\CR$ is compact with respect to $\tau_S$), and thus 
$\mathcal{M}_0(E) = \ell^1(\CR )$, the space of all discrete measures on $\CR$. 

From these observations it is easy to see that $T(t)'\mathcal{M}_0(E) \subset \mathcal{M}_0(E)$ -- hence
$\mathbf{T}$ is a semigroup on $(C_b(E), \mathcal{M}_0(E))$ -- and that $\mathbf{T}$ is $\sigma$-continuous at 0. 
Let us show that $C(E)$ is complete with respect to the compact-open topology. As noted in Example
\ref{ex.cb} this implies that the strict topology on $C_b(E)$ is complete and thus assertion (ii) follows
from Theorem \ref{t.integrablecb}.

So let $(f_\alpha) \subset C_b(E)$ be a net converging to some function $f$ with respect to the compact-open
topology. Fix $t\in \CR$. To prove that $f \in C_b(E)$, it suffices to prove that $f(t_n) \to f(t)$
as $n\to \infty$ for every sequence $t_n \downarrow t$. However, given such a sequence, the set 
$K=\{t, t_n\,:\,n\in \CN\}$ is $\tau_S$-compact and thus $f_\alpha \to f$ uniformly on $K$. 
The convergence of $f(t_n) \to f(t)$
as $n\to \infty$ now follows from a standard $\frac{\eps}{3}$-argument.

Concerning assertion (iii), we note that for every $f \in C_b(E)$ we have
\[ \int_0^1\dual{f}{T(t)'\delta_0}\, dt = \dual{f}{\lambda_{(0,1)}} \,\, ,\]
where $\lambda_{(0,1)}$ denotes the restriction of Lebesgue measure to $(0,1)$. Since this measure
does not belong to $\mathcal{M}_0(E)$, the orbit of $T(\cdot )'\delta_0$ is not locally $C_b(E)$-integrable
and (iii) follows from Theorem \ref{t.measurabledescription}.
\end{proof}

We close this section by proving that if the topology of $E$ is induced by a separable metric,
in particular if $E$ is a Polish space (i.e.\ the topology of $E$ is induced by a 
complete, separable metric), then the norming dual pair $(C_b(E), \mathcal{M}_0(E))$
is countably separated. As a consequence of this, Theorem \ref{t.countablysep} may be applied, yielding that
every integrable semigroup on $(C_b(E), \mathcal{M}_0(E))$ is uniquely determined by its Laplace transform.
Furthermore, if $E$ is a Polish space, then, given exponential boundedness, the $\sigma$-continuity assumption
in Theorem \ref{t.integrablecb} may be weakened to the requirement that $t\mapsto \dual{T(t)f}{\mu}$ is measurable
for every $f\in C_b(E)$ and $\mu \in \mathcal{M}_0(E)$. This is evident from the proof of that theorem.

\begin{thm}\label{t.countablysep}
Let $(E,\mathcal{B}(E))$ be a separable metric space endowed with its Borel $\sigma$-algebra.
Then the norming dual pair $(C_b(E),\mathcal{M}_0(E))$ is countably separated.
\end{thm}

\begin{proof}
Let $D:= \{x_m\, :\,m \in \CN \}$ be a countable, dense subset of $E$. Then
$\{\delta_{x_m}\, : m \in \CN\} \subset \mathcal{M}_0(E)$ separates points in $C_b(E)$ as
continuous functions which coincide on a dense subset are equal. To find a sequence in $C_b(E)$ 
which separates points in $\mathcal{M}_0(E)$, we proceed as follows. For $n,m\in \CN$, choose
$f_{n,m}\in C_b(E)$ such that
\[\one_{\overline{B}(x_m, \frac{1}{n+1})} \leq f_{n,m}\leq \one_{B(x_m, \frac{1}{n})^c} \,\, .\]
If $J\subset \CN$ is a finite subset, we put $f_{n,J} := \max\{ f_{n,m}\, : \, m \in J \}$ and define
\[ M := \{f_{n,J}\, : \, n \in \CN\, , \, J\subset \CN\,\, \mbox{finite}\, \} \,\, .\]
Then $M$ is a countable set. We claim that $M$ separates points in $\mathcal{M}_0(E)$. 
To that end, let $\mu \in \mathcal{M}_0(E)$ satisfy $\int f \, d\mu = 0$ for all $f \in 
M$. We have to prove that $\mu = 0$. Since $\mu$ is a Radon measure, it suffices to prove 
that $\mu (K) = 0$ for all compact sets $K$. So let a compact set $K\neq \emptyset$ be given. As $D$
is dense in $E$, the set $K$ is covered by $\{ B(x_m, (n+1)^{-1})\, : \, m \in \CN\, \}$
for every $n\in \CN$. Since $K$ is compact, there exist $m_1,\ldots,m_{k_n}$ such that $K$
is already covered by $\mathcal{B}_n := \{ B(x_{m_i}, (n+1)^{-1})\,:\, i= 1,\ldots, k_n\,\}$.
We may assume without loss that every ball in $\mathcal{B}_n$ intersects $K$. Define
$f_n := f_{n,\{m_1,\ldots,m_{k_n}\}} \in M$. Then $(f_n)_{n\in\CN}$ is a bounded sequence which
converges pointwise to $\one_K$. As $\int f_n\, d\mu \equiv 0$ by assumption, the dominated
convergence theorem yields $\mu (K) = \lim f_n\, d\mu = 0$. As $K$ was arbitrary, $\mu = 0$.  
\end{proof}

\subsection*{Acknowledgement} The author is grateful to M.\ Haase and J.\ van Neerven for some
valuable comments on earlier versions of this paper. We also thank J.\ Bonet and B.\ Cascales
for providing an early preprint of \cite{boca}.

\end{document}